\documentclass[letterpaper, 10 pt, conference]{ieeeconf}  

\IEEEoverridecommandlockouts                              
\usepackage{graphicx} 
\usepackage{amsmath} 
\usepackage{amssymb}  
\usepackage{array}
\usepackage{multirow}

\newtheorem{thm}{Theorem}
\newtheorem{proposition}[thm]{Proposition}

\newtheorem{definition}{Definition}
\newtheorem{assumption}{Assumption}

\title{Total Variation of the Control and Energy of Bilinear Quantum Systems}

\author{
\authorblockN{Nabile Boussa\"{i}d\authorrefmark{1}, Marco Caponigro\authorrefmark{2} and Thomas Chambrion\authorrefmark{3}}
\authorblockA{\authorrefmark{1}Laboratoire de math\'ematiques, Universit\'e de Franche--Comt\'e, 25030 Besan\c{c}on, France\\
{\tt\small Nabile.Boussaid@univ-fcomte.fr}}
\authorblockA{\authorrefmark{2}\'Equipe M2N,
 Conservatoire National des Arts et M\'etiers, 75003 Paris, France \\
{\tt\small marco.caponigro@cnam.fr}}
\authorblockA{\authorrefmark{3}Universit\'e de Lorraine, Institut \'Elie Cartan de Lorraine, UMR 7502,  Vandoeuvre-­l\`es-­Nancy,F-­54506, France  \\
CNRS, Institut \'Elie Cartan de Lorraine, UMR 7502, Vandoeuvre-­l\`es-­Nancy, F-­54506,
France\\
Inria, CORIDA, Villers-l\`es-Nancy, F-54600, France\\
{\tt\small Thomas.Chambrion@univ-lorraine.fr}}
}

\begin{document}

\maketitle
\thispagestyle{empty}
\pagestyle{empty}

\begin{abstract}
In the present note, we give two examples of bilinear quantum systems 
showing good agreement between the total variation of 
the control and the variation of the energy of
solutions,  with bounded or unbounded coupling 
term. The corresponding estimates in terms of 
the total variation of the control appear to be  
optimal.
\end{abstract}

\section{INTRODUCTION}

\subsection{Control of quantum systems}
The state of a quantum system evolving in a Riemannian manifold $\Omega$ is
described by 
its \emph{wave function}, a point $\psi$ in $L^2(\Omega, \mathbf{C})$. When the
system is 
submitted to an electric field (e.g., a laser), the time evolution of the wave
function is 
given, under the dipolar approximation and neglecting decoherence,  by the
Schr\"{o}dinger 
bilinear 
equation:
\begin{equation}\label{EQ_bilinear}
\mathrm{i} \frac{\partial \psi}{\partial t}=(-\Delta  +V(x)) \psi(x,t) +u(t) 
W(x) 
\psi(x,t)
\end{equation}
where $\Delta$ is the Laplace-Beltrami operator on $\Omega$,  $V$ and $W$ are
real 
potential accounting for the properties of the free system and the control
field 
respectively, while 
the real function of the time $u$ accounts for the intensity of the laser. 

In view of applications (for instance in NMR), it is important to know whether and how
it is possible to choose a suitable control $u:[0,T]\to \mathbf{R}$ in order to steer 
(\ref{EQ_bilinear}) from a given initial state to a given target. This question has raised 
considerable interest in the community in the last decade. After the negative results of 
\cite{bms} and \cite{Turinici} excluding 
exact controllability on the natural domain of the operator $-\Delta +V$ when
$W$ is bounded, the first, and 
at this day the only one, description of the attainable set for an example of
bilinear quantum system 
was obtained by (\cite{beauchard,beauchard-coron}).  Further investigations of
the approximate controllability of (\ref{EQ_bilinear}) were conducted using
Lyapunov techniques  
(\cite{nersesyan, Nersy, beauchard-nersesyan,Mirrahimi, MR2168664, mirra-solo}) 
and geometric techniques (\cite{Schrod,Schrod2}).

\subsection{Various notions of energies}

 Quantum control is a trans-disciplinary field where
different communities  use the same word ``energy'' with possibly different meaning.

Mathematically, the energy of system~(\ref{EQ_bilinear}) is any norm in (a subspace of) 
$L^2(\Omega, \mathbf{C})$ and the energy for the control $u$ in any norm in the space of 
admissible controls. 
A recurrent issue when studying systems of the type of (\ref{EQ_bilinear}) is to obtain a 
priori estimates of the energy of the system in terms of some energy of the control. Such 
energy 
estimates are crucial for many reasons, both for mathematical and engineering purposes, 
including for instance the proof of the well-posedness of the system  and the regularity 
of the solutions \cite{katino}, or estimates of the distance between the original infinite 
dimensional systems and some its finite dimensional approximations (see 
Section~\ref{SEC_GGA}
below). 

Physically, the energy of the quantum system (\ref{EQ_bilinear}) with wave function $\psi$ 
is $E(\psi)=\int_{\Omega}\left \lbrack (-\Delta +V) \overline{\psi } \right \rbrack \psi~ 
\mathrm{d}\mu$.
The physical energy is therefore constant in time whenever the control $u$ is
zero. When the control $u$ is nonzero, and provided
suitable 
regularity hypotheses, the energy evolves as 
\begin{equation}\label{EQ_evolution_energy_blse}
\frac{\mathrm{d}E}{\mathrm{d}t}=2 u(t) \Im \left ( \int_{\Omega}\left \lbrack  
(\Delta+V)\overline{\psi} \right \rbrack W\psi~ \mathrm{d}\mu \right ).
\end{equation}
Note that the time derivative of the energy $E$ at time $t$ depends on the value $u(t)$ of 
the intensity of the external field  \emph{and} on the wave function $\psi(t)$.

A natural question is to relate the mathematical energy of the control with the physical 
energy of the system.

Standard candidates for these estimates are the $L^p$ norms $
 \|u\|_{L^p(0,T)}=\left ( \int_0^T |u(t)|^p \mathrm{d}t \right )^{\frac{1}{p}},
 $
 for some suitable $p>0$.
Indeed, many previous works addressed the problem of the optimal control of the system 
(\ref{EQ_bilinear}) for costs involving the 
$L^2$ norm of the control 
(see for instance \cite{PhysRevA.42.1065} or \cite{symeon}).
 The main reason for choosing the $L^2$ norm is the fact that the natural Hilbert 
 structure of $L^2$  
allows the use of the powerful tools of Hilbert optimization.
It is common belief that there is a natural relation of the $L^2$ norm of $u$ 
and the energy of the system. The note~\cite{energy}  showed that, in general, 
the $L^1$-norm provides more information on the evolution of the system {{than}}
other  $L^{p}$-norms for $p>1$.

\subsection{Framework and notations}\label{SEC_notations}
To take advantage of the powerful tools of the theory of linear operators, we reformulate 
the bilinear dynamics (\ref{EQ_bilinear}) in  more abstract framework. In the separable 
Hilbert space $H$, we consider the bilinear system 
\begin{equation}\label{EQ_main}
\frac{\mathrm{d}\psi}{\mathrm{d}t}(t)=A\psi(t) +u(t)B \psi(t)
\end{equation}
where the (time independent) linear operators $A$ and $B$ satisfy some regularity 
assumptions. 
 
\begin{assumption}\label{ASS_ass_general}
The triple $(A,B,\Phi)$ is such that
\begin{enumerate}
\item $A$ is skew-adjoint, possibly unbounded, on its domain $D(A)$;
\label{ASS_general_A_skew_adjoint}
\item $-\mathrm{i}A$ is positive;\label{ASS_general_bounded_below}
\item $B$ is bounded relatively to $A$: there exist $a$ and $b$ in $\mathbf{R}$ such that 
$\|B\psi\|\leq a\|A\psi\|+b\|\psi\|$;\label{ASS_general_B_relativ_bounded}
\item $\Phi=(\phi_j)_{j\in \mathbf{N}}$ is a Hilbert basis of $H$ made of eigenvectors of 
$A$: for every $j$ in $\mathbf{N}$, there exists $\lambda_j$ in $\mathbf{R}$ such that 
$A\phi_j=-\mathrm{i}\lambda_j \phi_j$;\label{ASS_general_Phi_basis}
\end{enumerate}
\end{assumption}
If $A$ and $B$ satisfy Assumption
\ref{ASS_ass_general}.\ref{ASS_general_B_relativ_bounded}, we denote 
\begin{multline*}
\|B\|_{A} = \inf\{a\in \mathbf{R} \mid \exists b \in
\mathbf{R} \\\mbox{ for which }  \|B\psi\|\leq a\|A\psi\|+b\|\psi\|,\; \forall
\psi \in D(A)\}.
\end{multline*}


It is known that if $(A,B,\Phi)$ satisfies Assumption \ref{ASS_ass_general}, then for 
every $u:[0,T_u]\to (-1/\|B\|_A,1/\|B\|_A)$ with bounded variation, there exists
a continuous mapping $t\mapsto \Upsilon^u_t$ 
taking value in the unitary group $\mathbf{U}(H)$ of $H$ such that, for every $\psi$ in 
$D(A)$, $t\mapsto \Upsilon^u_t\psi$ is differentiable almost everywhere and
satsifies (\ref{EQ_main}) for 
almost every $t$ in $(0,T]$. For a proof of this well-posedness result,  see 
\cite{Kato1953} for a general theory of time dependent (non-necessarily skew-adjoint) 
Hamiltonians or \cite{boussaid:hal-00784876} for an elementary proof adapted to the 
bilinear structure of (\ref{EQ_main}).

\begin{definition}
Let $(A,B,\Phi)$ satisfy Assumption \ref{ASS_ass_general}. The system $(A,B)$ is \emph{approximately controllable} if, for every $\psi_0,\psi_1$ in the unit Hilbert sphere, for every $\varepsilon>0$, there exists $u_{\varepsilon}:[0,T_{\varepsilon}]\rightarrow \mathbf{R}$ such that $\|\Upsilon^{u_\varepsilon}_{T_\varepsilon}\psi_0-\psi_1\|<\varepsilon$.
\end{definition}
The following sufficient criterion for approximate controllability is the central result of \cite{Schrod2}, centered on the notion of non-degenerate (or non-resonant) transitions.
\begin{definition}
Let $(A,B,\Phi)$ satisfy Assumption \ref{ASS_ass_general}. A pair $(j,k)$ of integers is a  
\emph{non-degenerate transition} of $(A,B,\Phi)$ if (i)  $\langle \phi_j,B\phi_k \rangle \neq 
0$ and (ii) for every $(l,m)$ in $\mathbf{N}^2$, $|\lambda_j-\lambda_k|=|
\lambda_l-\lambda_m|$ implies $(j,k)=(l,m)$ or $\langle \phi_l, B \phi_m\rangle =0$ or 
$\{j,k\}\cap \{l,m\}=\emptyset$.
\end{definition}
 \begin{definition}
Let $(A,B,\Phi)$ satisfy Assumption \ref{ASS_ass_general}. A subset $S$ of $\mathbf{N}^2$ 
is a \emph{non-degenerate chain of connectedness} of $(A,B,\Phi)$ if
(i) for every $(j,k)$ in $S$, $(j,k)$ is a non-degenerate transition of $(A,B)$ and (ii) 
for every $r_a,r_b$ in $\mathbf{N}$, there exists a finite sequence $r_a=r_0,r_1,
\ldots,r_p=r_b$ in $\mathbf{N}$ such that, for every $j\leq p-1$, $(r_j,r_{j+1})$ belongs 
to $S$.   
\end{definition}
\begin{proposition}
 Let $(A,B,\Phi)$ satisfy Assumption \ref{ASS_ass_general}.  If $(A,B)$ admits a non-degenerate chain of connectedness, then $(A,B)$ is approximately controllable. 
\end{proposition}

\subsection{Main result}

The contribution of this note is to show the good agreement between the total variation of 
the control and the variation of the $A$-norm of the wave function. The $A$-norm, equal to 
$\|A\psi\|$ for every $\psi$ in $D(A)$ is not equal, in general, to the energy 
$\||A|^{1/2}\psi\|$. However, if $\phi_j$ is an eigenvector of $A$  with associated
eigenvalue $-\mathrm{i}\lambda_j$, then
$\|A\phi_j\|=|\lambda_j|=\||A|^{1/2}\phi_j\|^2$. 

We have to distinguish between the cases where $B$ is bounded and when it is not. 

\subsubsection{Bounded case}

When $B$ is bounded, the growth of the $A$ norm of $\Upsilon^u_t\psi$ is at most linear 
with respect to the total variation of the control (see Section 
\ref{SEC_estimee_bounded}). We present, in  Section~\ref{SEC_example_bounded}, an example for which the growth is indeed linear. 
More precisely, we will show the following.
\begin{proposition}\label{PRO_main_bounded}
There exists $(A,B,\Phi)$ satisfying Assumption \ref{ASS_ass_general} with $B$ bounded such that, for 
every $M$ in $\mathbf{R}$, there exists $u_M:[0,T_M]\to \mathbf{R}$ with bounded variation with $  \|A \Upsilon^{u_M}_{T_M} \phi_1\| \geq M$ and 
${ M \geq \frac{\|B\|}{4} TV_{[0,T_M]}(u_M)}$.
\end{proposition}

\subsubsection{Unbounded case}
When $B$ is unbounded, the growth of the $A$ norm of $\Upsilon^u_t\psi$ is at most 
exponential with respect to the total variation of the control (see Section 
\ref{SEC_estimee_unbounded}). We present, in Section~\ref{SEC_example_unbounded}, an example for which the growth is indeed 
exponential. 
More precisely, we will show the following.
\begin{proposition}\label{PRO_main_unbounded}
There exists  a triple $(A,B,\Phi)$ satisfying Assumption \ref{ASS_ass_general} with $B$ 
unbounded such that, for 
every $M$ large enough in $\mathbf{R}$, there exists $u_M:[0,T_M]\to \mathbf{R}$ with bounded variation with $  \|A \Upsilon^{u_M}_{T_M} \phi_1\| \geq M$ and 
$\displaystyle{ M \geq 4 \exp \left (\sqrt{\frac{2}{3}}\|B\|_A TV_{[0,T_M]}(u_M) \right
)-6}$.
\end{proposition}

\subsection{Content of the paper}
In Section~\ref{SEC_various_estimates}, we review some classical estimates for the growth 
of  $|A|^r$-norms of the wave function in terms of $L^p$ norms 
(Section~\ref{SEC_Lp_estimates}) and total variation (Section~\ref{SEC_BV_estimates}) of 
the control. Some examples of use of these estimates for the approximation of the infinite 
dimensional system (\ref{EQ_main}) by its finite dimensional approximations are given in 
Section \ref{SEC_GGA}. Section~\ref{SEC_periodic} is a quick survey of basic facts about 
averaging theory for finite dimensional bilinear systems.  These convergence results will 
be instrumental in Section \ref{SEC_Examples} to prove Proposition \ref{PRO_main_bounded} 
(Section~\ref{SEC_example_bounded}) and Proposition~\ref{PRO_main_unbounded} 
(Section~\ref{SEC_example_unbounded}).

\section{SOME ENERGY ESTIMATES}\label{SEC_various_estimates}

\subsection{Weakness of $L^p$ estimates}\label{SEC_Lp_estimates}

 Let $(A,B,\Phi)$ satisfy Assumption \ref{ASS_ass_general} and admit a non-degenerate chain of connectedness. For every $r>0$, for every $j,k$ in $\mathbf{N}$ and $\varepsilon>0$
we define
 $ \mathcal{A}_r^{\varepsilon}(j,k)$ as the set of  functions $u:[0,T_u]\rightarrow 
 \mathbf{R}$ in $L^1([0,T_u])\cap L^r([0,T_u])$ such that $\|
 \Upsilon^{u}_{T_u}\phi_j-\phi_k\|<\varepsilon$.
We consider the quantity 
$$
\mathcal{C}_r(\phi_{j},\phi_{k})=\sup_{\varepsilon>0} \left ( \inf_{ u \in 
\mathcal{A}_r^{\varepsilon}(j,k)} \|u\|_{L^r(0,T_u)} \right).
$$  
This quantity is the infimum of the $L^{r}$-norm of a control achieving approximate 
controllability. It clearly satisfies the triangle inequality.
Next proposition states that $\mathcal{C}_{r}$ is a distance on the space of eigenlevels 
only when $r=1$. 
Its proof is given in \cite{energy}.
\begin{proposition}\label{PRO_distance_L1}
$\mathcal{C}_1$ is a distance on the set $\{\phi_j,j\in \mathbf{N}\}$. 
For $r>1$, $\mathcal{C}_r$ is equal to zero on the set $\{\phi_j,j\in \mathbf{N}\}$.
\end{proposition}
Proposition \ref{PRO_distance_L1} illustrates various flaws of $L^p$ estimates
of system 
(\ref{EQ_main}). First, and contrary to the immediate intuition, $L^p$ norms with $p>1$ 
(and in particular the $L^2$ norm) do not permit to distinguish among the energy levels of 
$A$. Precisely, if $(A,B,\Phi)$ admits a non-degenerate chain of connectedness, for every 
non empty open set $\mathcal V$ in $\mathbf{S}_H$, there exists $u:[0,T]\to 
(-1/a,1/a)$ in $L^p([0,T])$ with $\|u\|_{L^p([0,T])}$ arbitrarly small   
such that $\Upsilon^u_T\phi_1 $ belongs to $\mathcal V$, see 
\cite{boussaid:hal-00710040}.

While the $L^1$ norm allows to distinguish among the energy levels of $A$, the distance 
${\mathcal C}_1$ depends only on $B$ and the non-degenerate chains of connectedness of 
$(A,B,\Phi)$ (and not on the eigenvalues of $A$). For instance, the computation of  ${\mathcal C}_1(\phi_1,\phi_2)$ done in Section IV of \cite{energy} remains valid, with unchanged result, if one replaces $A$ by $\pm \mathrm{i}|A|^k$ for any positive integer $k$.

\subsection{Estimates based on total variation}\label{SEC_BV_estimates}

The following estimates can be deduced from the general theory due to Kato 
\cite{Kato1953}. They are valid in context much broader than Assumption 
\ref{ASS_ass_general}. In particular, there is no need for $H$ to admit a 
Hilbert basis made of eigenvectors of $A$. 

In the following we will impose $u(0)=0$. This is always the case if 
one 
replaces $A$ by $A+u(0)B$. Moreover if $(A,B, \Phi)$ satisfies 
Assumption~\ref{ASS_ass_general} then there exists $\Phi'$ and $b\in\mathbf{R}$ 
such 
that $(A+u(0)B-\mathrm{i}b,B, \Phi')$ satisfies 
Assumption~\ref{ASS_ass_general} as well.

\subsubsection{Bounded case}\label{SEC_estimee_bounded}
\begin{proposition}
Let $(A,B,\Phi)$ satisfy Assumption \ref{ASS_ass_general} with $B$ bounded. Then, for 
every $u:
[0,T]\to \mathbf{R}$ with bounded variation and $u(0)=0$, for every $j$ in 
$\mathbf{N}$, $\|A 
\Upsilon^u_{T} \phi_j\| -\|A \phi_j\| \leq 2\|B\| TV_{[0,T]}(u)$. 
\end{proposition}
\begin{proof}
 Notice that if $u(0)=0$ then $|u(T)|\leq TV_{[0,T]}(u)$.  Hence it is 
enough to prove for every $j$ in $\mathbf{N}$ that $\|(A +u(T)B)
\Upsilon^u_{T} \phi_j\| -\|A \phi_j\| \leq \|B\| TV_{[0,T]}(u)$.

Any bounded variation function can be approximated pointwise by a sequence 
of piecewise constant functions $(u_n)_{n\in \mathbf{N}}$ such that $|u_n| \leq 
|u|$ and $TV_{[0,T]}(u_n)\leq TV_{[0,T]}(u)$. 

Following~\cite{boussaid:hal-00784876}, we have that
$\Upsilon^{u_n}_{T} \phi_j\to \Upsilon^u_{T} \phi_j$. Thus it is sufficient to prove the statement  
for piecewise constant controls.  
The proof for piecewise constant controls follows from the estimate
$\|(A +uB) \exp{(t(A +uB))} \phi\| -\|A \phi\| \leq \|B\| |u|$.
Indeed for a piecwise constant function the associated $\Upsilon^u_{t}$ is a 
product of $\exp{t_i(A +u_iB)}$ for different values of $u_i$ and $t_i$. 
The details of the proof are similar to those of
\cite[Section 2]{boussaid:hal-00784876}.
\end{proof}

\subsubsection{Unbounded case}\label{SEC_estimee_unbounded}
\begin{proposition}
Let $(A,B,\Phi)$ satisfy Assumption \ref{ASS_ass_general} with $B$ unbounded. Then, for 
every $0<\delta<1$, for every $u:[0,T]\to 
(-(1-\delta)/\|B\|_A,(1-\delta)/\|B\|_A)$ with bounded 
variation and $u(0)=0$, for every $\psi$ in $D(A)$, $\|A \Upsilon^u_{T} \psi\| 
\leq e^{\|B\|_A TV_{[0,T]}
(u)/\delta}  \|A \psi\| $. 
\end{proposition}

The proof in the unbounded case, which can be found in 
\cite[Proposition~3]{boussaid:hal-00784876}, follows the lines of the bounded case.

\subsection{Good Galerkin Approximations}\label{SEC_GGA}

\begin{assumption}\label{ASS_weakly_coupled}
The quadruple $(A,B,\Phi,k)$ is such that
\begin{enumerate}
\item $(A,B,\Phi)$ satisfies Assumption \ref{ASS_ass_general};
\item $-\mathrm{i}A$ is positive;
\item $k$ is a positive real number;
\item for every $u$ in $\mathbf{R}$, the domains $D(|A+uB|^k)$ of $|A+uB|^k$ and $D(|A|
^k)$ of $|A|^k$ coincide;
\item there exists $d,r$ in $\mathbf{R}$, $r<k$ such that $\|B \psi \| < d \||A|^r \psi \|
$ for every $\psi$; 
\item the supremum $c_k(A,B)$ of the subset of $\mathbf{R}$  $\{|\Re \langle 
|A|^k \psi, B 
\psi \rangle| /|\langle |A|^k \psi,  \psi \rangle|, \psi \in D(|A|^k)\}$ is 
finite.  
\end{enumerate}
\end{assumption}

For every $N$ in $\mathbf{N}$, we define the orthogonal projection
$$
 \pi_N:\psi \in H \mapsto \sum_{j\leq N} \langle \phi_j,\psi\rangle
\phi_j \in H.
$$

\begin{definition}
Let $N \in \mathbf{N}$.  The \emph{Galerkin approximation}  of \eqref{EQ_main}
of order $N$ is the system in $H$
\begin{equation}\label{eq:sigma}
\dot x = \left(A^{(N)} + u B^{(N)}\right) x 
\tag{$\Sigma_{N}$}
\end{equation}
where $A^{(N)}=\pi_N A \pi_N$ and $B^{(N)}=\pi_N B \pi_N$ are the
\emph{compressions} of $A$ and $B$ (respectively).
\end{definition}

We denote by $X^{u}_{(\Phi,N)}(t,s)$ the propagator of \eqref{eq:sigma}.
\begin{definition}
The system $(A,B,\Phi)$ admits a sequence of \emph{Good Galerkin Approximations} (GGA
in short), in time $T\in (0,+\infty]$, in a subspace  $D$ (with
norm $\|\cdot \|_D$) of $H$, in terms of a functional 
norm $N(\cdot)$ on a functional space $\mathbf{U}$  if, 
for any $K,\varepsilon>0$, for any $\psi$ in $D$, there exists $N$ in
$\mathbf{N}$ such that,  for any $u$ 
in $\mathbf U$, $N(u)\leq K$ implies
$\|(X^{u}_{(\mathbf{\Phi},N)}(t,0)-\Upsilon^u_{t,0})\psi \|_D
<\varepsilon$ for any $t<T$.
\end{definition}

\begin{proposition}\label{PROP_GGAWeaklyCoupled}
Let $(A,B,\Phi,k)$ satisfy Assumption \ref{ASS_weakly_coupled}. Then $(A,B,\Phi)$ admits a 
sequence of good Galerkin approximations in infinite time, in $D(A)$  in terms of $L^1$ 
norm for locally integrable controls.
\end{proposition}
Last proposition is proved in~\cite{weakly-coupled} for piecewise constant 
controls. The generalization to $L^1$ controls follows from~\cite{boussaid:hal-00784876}.

\begin{proposition}\label{PROP_GGA}
Let $d>0$, $r<1$ and $(A,B,\Phi)$ satisfy Assumption \ref{ASS_ass_general} with $\|B\psi\|
\leq d\||A|^r \psi \|$ for every $\psi$ in $D(|A|^r)$. Then $(A,B,\Phi)$ admits 
a 
sequence of good Galerkin approximations in infinite time, in $D(A)$  in terms 
of $TV+L^1$ norm for controls with bounded variation.
\end{proposition}

This proposition is proved in~\cite{boussaid:hal-00784876}. Notice, that if imposing $u(0)=0$ for the control term then the $L^1$ norm of the control over any finite time interval is bounded by a multiple of the total variation $TV$.

\section{PERIODIC CONTROL LAWS OF BILINEAR QUANTUM SYSTEMS}\label{SEC_periodic}
\subsection{Averaging theory}\label{SEC_Averaging_theory}
The mathematical concept of averaging of dynamical systems was introduced more than a 
century ago and has now developed into a well-established theory, see for instance the 
books of Guckenheimer \& Holmes \cite{Holmes}, Bullo \& Lewis \cite{Bullo} or Sanders, 
Verhulst \& Murdock \cite{Sanders}. It was observed that, for regular $F$ and small 
$\varepsilon$, the trajectories of the system $\dot{x}=\varepsilon F(x,t,\varepsilon)$
remain $\varepsilon$ close, for time of order $1/\varepsilon$, to the trajectories of the 
average system $\dot{x}=\widetilde{F}(x)$ where $\widetilde{F}(x)=\lim_{t\to \infty} 1/t 
\int_0^t F(x,t,0)$.

In quantum physics, this concept of averaging is used intensively to transfer a system of 
type (\ref{EQ_main}) from an eigenstate of $A$ associated with eigenvalue 
$-\mathrm{i}\lambda_j$ to another associated with eigenvalue $-\mathrm{i}\lambda_k$ with a 
periodic control with small enough amplitude and frequency $|\lambda_j-\lambda_k|$.
 
The following results is proved in \cite{periodic}.
\begin{proposition} \label{PRO_periodic}
Let $(A,B,\Phi)$ satisfy Assumption \ref{ASS_ass_general}.
Assume that $(j,k)$ is a non-degenerate transition of $(A,B,\Phi)$. 
Define $T=2\pi/|\lambda_j-\lambda_k|$ and
let $u^\ast:\mathbf{R}\to (-1/\|B\|_A,1/\|B\|_A)$ be $T$-periodic and with bounded variation on 
$[0,T]$.
If $\displaystyle{\int_0^T \!\!\! u^{\ast} (\tau) e^{  \mathrm{i}
(\lambda_{j}-\lambda_{k})\tau} \mathrm{d}\tau \neq 0}$ and $\displaystyle{\int_0^T \!\!\! u^{\ast} (\tau) e^{  \mathrm{i}
(\lambda_{l}-\lambda_{m})\tau} \mathrm{d}\tau = 0}$ for every $(l,m)$ such that (i) 
$\{j,k\}\neq \{l,m\}$, (ii)  $\{j,k\} \cap \{l,m\}\neq \emptyset$, (iii) $|
\lambda_l-\lambda_m|\in (\mathbf{N}\setminus\{1\}) |\lambda_j-\lambda_k|$ and 
(iv) $
b_{lm}\neq 0$,
  then, for every $n$ in $\mathbf{N}$, there exists $T_n^{\ast}$ in $(nT^{\ast}-
  T,nT^{\ast}+T)$ such that $|\langle 
\phi_k,X^{u}_{(\Phi,N)}(T^{\ast}_n,0) \phi_j \rangle |$ tends to $1$ 
as $n$ tends to infinity, with
$\displaystyle{T^{\ast}=  \frac{\pi  T}{2 |b_{j,k}|  \left |\int_0^T \!\! \!u^{\ast}(\tau)e^{\mathrm{i}
(\lambda_{j}-\lambda_{k}) \tau}  \mathrm{d}\tau \right |},\quad I=\int_0^T \! \!|u^{\ast}
(\tau)|\mathrm{d}\tau,}$
$$ K=\frac{IT^{\ast}}{T}  \mbox{ and } C=\sup_{(j,k)\in \Lambda} \left | 
\frac{\int_0^T u^{\ast}(\tau) 
e^{\mathrm{i} (\lambda_l-\lambda_m)\tau}\mathrm{d}\tau}{\sin \left ( \pi\frac{|
\lambda_l-\lambda_m|}{|\lambda_j-\lambda_k|} \right )} \right |,$$
where $\Lambda$ is the set of all pairs $(l,m)$ in  $\{1,\ldots,N\}^2$  such that $b_{lm} 
\neq 0$ and $\{l,m\}\cap\{j,k\} \neq \emptyset$ and $ |\lambda_l-\lambda_m|\notin 
\mathbf{Z}|\lambda_2-\lambda_1|$.
\end{proposition}

Notice that Proposition~\ref{PRO_periodic} does not claim that 
$\|A(X^{u}_{(\Phi,N)}(T^{\ast}_n,0)\phi_j-\phi_k)\|$ 
tends to zero as $n$ tends to infinity. However, 
\begin{multline*}
\liminf_{n\to \infty} \|A 
X^{u}_{(\Phi,N)}(T^{\ast}_n,0)\phi_j\|\\\geq 
\liminf_{n\to \infty} \lambda_k |\langle \phi_k, X^{u}_{(\Phi,N)}(T^{\ast}_n,0)
\phi_j\rangle |=\lambda_k. 
\end{multline*}

Using Propostion~\ref{PROP_GGAWeaklyCoupled} or Proposition~\ref{PROP_GGA} 
these can be extended to inifinite dimensional system \eqref{EQ_main} with 
$(A,B,\Phi,k)$ satisfing Assumption \ref{ASS_weakly_coupled} or Assumption 
\ref{ASS_ass_general} with $\|B\psi\|
\leq d\||A|^r \psi \|$ for every $\psi$ in $D(|A|^r)$ for some $d>0$ and $r<1$.

\subsection{Averaging using the sine function}\label{SEC_averaging_sinus}

Let $(A,B,\Phi)$ satisfy Assumption \ref{ASS_ass_general} and $(j,k)$ be a non-degenerate 
transition of $(A,B,\Phi)$.  We define 
$\omega=|\lambda_j-\lambda_k|$. We apply Proposition \ref{PRO_periodic} with 
$u^{\ast}:t\mapsto \sin(\omega t)$. 
For $n$ large enough, $\|u^\ast(t)/n \|_{L^\infty}\leq 1/\|B\|_A$. Straightforward 
computations give
$$T=\frac{2\pi}{\omega},\quad  T^\ast=\frac{\pi}{|b_{jk}|}, \quad
I=\frac{4}{\omega},$$
and we compute
\begin{eqnarray*}
\lefteqn{TV_{[0,T_n^\ast)}\left (\frac{u^{\ast}}{n} \right ) =
\frac{1}{n}\int_0^{T_n^\ast}\omega 
|\cos(\omega t)|\mathrm{d}t}\\
&=&\frac{\omega}{n} \left ( \int_0^{\left \lfloor \frac{nT^{\ast}_n}{T} \right \rfloor T}
\!\!\! |\cos(\omega t)|\mathrm{d}t +\right. 
 \left. \int_{\left \lfloor \frac{nT^{\ast}_n}{T} 
\right \rfloor T} ^{T_n^\ast} \!\!\! |\cos(\omega t)|\mathrm{d}t \right ).
\end{eqnarray*}
As $n$ tends to infinity, $\displaystyle{\frac{\omega}{n}\int_{\left \lfloor 
\frac{nT^{\ast}_n}{T} 
\right \rfloor T} ^{T_n^\ast} \!\!\! |\cos(\omega t)|\mathrm{d}t}$ tends to zero, hence
\begin{eqnarray*}
\lim_{n\to \infty}TV_{[0,T_n^\ast)}\left (\frac{u^{\ast}}{n} \right ) &=&\lim_{n\to 
\infty} \frac{\omega}{n}\left \lfloor \frac{nT^{\ast}_n}{T} \right \rfloor \int_0^T \!\!\! 
|\cos(\omega t)|\mathrm{d}t\\
&=&\lim_{n\to \infty} 4\frac{\omega T_n^\ast}{n}=\frac{2\omega}{|b_{jk}|}.
\end{eqnarray*}

\section{EXAMPLES}\label{SEC_Examples}
\subsection{The bounded case: 2D rotation of a linear molecule}\label{SEC_example_bounded}
Consider a linear molecule whose only degree of freedom is the planar rotation, in a 
fixed plan, about its fixed center of mass.  This system has been thoroughly studied (see the references given in \cite{noiesugny-CDC} or \cite{weakly-coupled} for instance). 

In this model, the Schr\"{o}dinger equation 
reads
\begin{equation}\label{EQ_rotation2D}
 \mathrm{i}\frac{\partial \psi}{\partial t}=-\Delta \psi +\cos\theta \psi, \quad \theta 
 \in \Omega,
\end{equation}
$\Omega=\mathbf{R}/2\pi\mathbf{Z}$ is the unit circle endowed with the Riemannian 
structure inherited from 
$\mathbf{R}$, $H$ is the space of odd functions of $L^2(\Omega,\mathbf{C})$, 
$A=\mathrm{i}\Delta$ ($\Delta$ is the restriction to $H$ of the Laplace-Beltrami operator 
of $\Omega$) and $B:\psi \mapsto (\theta \mapsto \cos(\theta)\psi(\theta))$ is the 
multiplication by cosine. 
 
 In the Hilbert basis $\Phi=(\theta \mapsto \sin(k\theta))_{k \in \mathbf{N}}$ of $H$, $A$ 
is diagonal with diagonal $-\mathrm{i}k^2, k=1\ldots\infty$ and $B$ is tri-diagonal with 
$b_{k,k}=0, b_{k,k+1}=-\mathrm{i}/2$, $b_{j,k}=0$ for every $k,j$ in $\mathbf{N}$ such 
that $|j-k|>1$. 

The triple $(A,B,\Phi)$ satisfies Assumption \ref{ASS_ass_general}, $B$ is bounded and $\|
B\psi\|\leq 0\|A\psi\|+\sqrt{2}\|\psi\|$ for every $\psi$ in $H$. The set $\{(k,k+1),k\in \mathbf{N}\}$ is a non-degenerate chain of connectedness for $(A,B,\Phi)$.

For every $j,n_j$ in $\mathbf{N}$, we define the control 
$u^{\ast,j,n_j}:t\in[0,2n_j\pi]\mapsto \sin((2j+1)t)/n_j$, and
for every $N$ in $\mathbf{N}$, we define $u^{\ast,(n_1,n_2,\ldots,n_{N-1})}$ by the 
concatenation of $u^{\ast,1,n_1}$, $u^{\ast,2,n_2}$, $\ldots$, $u^{\ast,N-1,n_{N-1}}$. 

By Proposition \ref{PRO_periodic}, $$\liminf_{n_1,\ldots,n_{N-1} \to \infty}\|
A\Upsilon^{u^{\ast,(n_1,n_2,\ldots,n_{N-1})}}\phi_1\|\geq \lambda_{N}=N^2.$$

From Section \ref{SEC_averaging_sinus}, we compute
\begin{eqnarray*}
\lefteqn{\liminf_{n_1,n_2,\ldots,n_{N-1}\to \infty}\!\!\!\! TV_{[0,2\pi(n_1+n_2+\ldots+n_{N-1}]}(u^{\ast,(n_1,n_2,\ldots,n_{N-1})})}\\
&~~~~~~~~~~~~~~~~~~~~~~~~~~~~~~~~~~=&\sum_{j=1}^{N-1} 4(2j+1)=4N^2,
\end{eqnarray*}
which proves Proposition \ref{PRO_main_bounded}.
\subsection{The unbounded case: perturbation of the harmonic oscillator}\label{SEC_example_unbounded}
The second model we consider is a perturbation of the quantum harmonic oscillator, with 
dynamics give by
\begin{equation}\label{EQ_harmonic_oscillator}
\mathrm{i} \frac{ \partial \psi}{\partial t }= \left \lbrack (-\Delta+x^2)+ 
(-\Delta+x^2)^{-1} \right \rbrack \psi +u(t) x^2 \psi. 
\end{equation} 
With the notations of Section \ref{SEC_notations}, $H$ is the Hilbert space of the odd 
functions of $L^2(\mathbf{R},\mathbf{C})$, $A=-\mathrm{i}\left \lbrack (-\Delta+x^2)+ 
(-\Delta+x^2)^{-1} \right \rbrack $ where $\Delta$ is the restriction of the Laplacian to 
the space of odd functions and $B$ is the multiplication, in $H$ by $-\mathrm{i}x^2$.
Denoting by $H_n$ the $n^{th}$ Hermite function, we check that 
$A H_{2n-1}=-\mathrm{i}((4n-1)+(4n-1)^{-1})H_{2n-1}$, hence $\Phi=(H_{2n-1})_{n\in 
\mathbf{N}}$ is a 
Hilbert basis of $H$ made of eigenvectors of $A$. 
Moreover, $B H_1=- \mathrm{i} (1/2 H_1 +\sqrt{3/2} H_3)$ and, 
 for every $n$ in $\mathbf{N}$, $n\geq 2$,
\begin{eqnarray*}
\lefteqn{B H_{2n-1}=-\mathrm{i}\left \lbrack \sqrt{n\left (n-\frac{1}{2}\right )}H_{2n-3}+\left (n-\frac{1}{2} \right ) H_{2n-1} \right.} \\
 &&\quad \quad  \quad \quad  \quad+ \left. \sqrt{n\left (n+\frac{1}{2}\right )} H_{2n+1}\right \rbrack.\quad \quad \quad \quad \quad \quad \quad 
\end{eqnarray*}
In the basis $\Phi$, $A$ is diagonal with diagonal entries $(-\mathrm{i}((4n-1)+(4n-1)^{-1}))_{n\in \mathbf{N}}$ and 
$B$ is tri-diagonal. 
The system $(A,B,\Phi)$ is tri-diagonal in the sense of  
\cite{boussaid:hal-00784876} and satisfies Assumption \ref{ASS_ass_general} with 
$\|B\|_A \leq \frac{\sqrt{6}}{4}$ (Proposition 12 of 
\cite{boussaid:hal-00784876} applied with $r=1$ and $C=1/4$). 

For every $j,n_j$ in $\mathbf{N}$, we define the control 
$u^{\ast,j,n_j}:t\in[0,2n_j\pi]\mapsto \sin((2j+1)t)/n_j$, and
for every $N$ in $\mathbf{N}$, we define $u^{\ast,(n_1,n_2,\ldots,n_{N-1})}$ by the 
concatenation of $u^{\ast,1,n_1}$, $u^{\ast,2,n_2}$, $\ldots$, $u^{\ast,N-1,n_{N-1}}$. 

By Proposition \ref{PRO_periodic}, $$\liminf_{n_1,\ldots,n_{N-1} \to \infty}\|
A\Upsilon^{u^{\ast,(n_1,n_2,\ldots,n_{N-1})}}\phi_1\|\geq \lambda_{N}\geq 4 N-1.$$

From Section \ref{SEC_averaging_sinus}, we compute, similarly to what we have done in Section \ref{SEC_estimee_bounded}, 
\begin{eqnarray*}
\lefteqn{\liminf_{n_1,n_2,\ldots,n_{N-1}\to \infty}\!\!\!\!\! TV_{[0,2\pi(n_1+n_2+\ldots+n_{N-1}]}(u^{\ast,(n_1,n_2,\ldots,n_{N-1})})}\\
&~~~~~~~~~~~~~~~~~~~~~~~~~~~~~~~~=&\sum_{j=1}^{N-1} \frac{2}{N}\\
&~~~~~~~~~~~~~~~~~~~~~~~~~~~~~~~~~\leq & 2\log(N+1),\quad \quad \quad \quad 
\end{eqnarray*}
which proves Proposition \ref{PRO_main_unbounded}.
\addtolength{\textheight}{-20.5cm}
\section{CONCLUSIONS}

\subsection{Contribution}
We  exhibited two examples showing that the Kato estimates for the $A$-norm of the 
solutions of a  bilinear quantum system,  with bounded or unbounded coupling term,  are 
optimal up to a multiplicative constant. These estimates are given in terms of the total 
variation of the control. 

\subsection{Perspectives}
An interesting, and probably difficult, question is the optimal control of bilinear 
quantum systems when the cost is the total variation of the control. Approximation 
procedures (as the Good Galerkin Approximations presented in this note) allow to consider 
only a finite dimensional problem. The main difficulty will come from the non-smoothness of the cost (total variation) which will lead to the use of tools of non-smooth analysis.
\section{ACKNOWLEDGEMENTS}

This work has been partially supported by INRIA Nancy-Grand Est, by
French Agence National de la Recherche ANR ``GCM'' program
``BLANC-CSD'', contract number NT09-504590 and by European Research
Council ERC StG 2009 ``Ge\-Co\-Methods'', contract number
239\-748.

It is a pleasure for the third author to thank Ugo Boscain for having attracted his attention to this question.

\bibliographystyle{IEEEtran}
\bibliography{biblio}

\end{document}